\documentclass[]{amsart}
\usepackage{amscd,amsthm,amssymb,amsfonts,amsmath,euscript}

\theoremstyle{plain}
\newtheorem{thm}{Theorem}[section]
\newtheorem{lemma}[thm]{Lemma}
\newtheorem{prop}[thm]{Proposition}
\newtheorem{cor}[thm]{Corollary}

\theoremstyle{definition}
\newtheorem{defn}[thm]{Definition}

\theoremstyle{remark}
\newtheorem{remark}[thm]{Remark}

\newcommand{\nc}{\newcommand}

\def\makeop#1{\expandafter\def\csname#1\endcsname
  {\mathop{\rm #1}\nolimits}\ignorespaces}
\makeop{Hom}   \makeop{End}   \makeop{Aut}   \makeop{Isom}  \makeop{Pic} 
\makeop{Gal}   \makeop{ord}   \makeop{Char}  \makeop{Div}   \makeop{Lie} 
\makeop{PGL}   \makeop{Corr}  \makeop{PSL}   \makeop{sgn}   \makeop{Spf}
\makeop{Spec}  \makeop{Tr}    \makeop{Nr}    \makeop{Fr}    \makeop{disc}
\makeop{Proj}  \makeop{supp}  \makeop{ker}   \makeop{im}    \makeop{dom}
\makeop{coker} \makeop{Stab}  \makeop{SO}    \makeop{SL}    \makeop{SL}
\makeop{Cl}    \makeop{cond}  \makeop{Br}    \makeop{inv}   \makeop{rank}
\makeop{id}    \makeop{Fil}   \makeop{Frac}  \makeop{GL}    \makeop{SU}
\makeop{Nrd}   \makeop{Sp}    \makeop{Tr}    \makeop{Trd}   \makeop{diag}
\makeop{Res}   \makeop{ind}   \makeop{depth} \makeop{Tr}    \makeop{st}
\makeop{Ad}    \makeop{Int}   \makeop{tr}    \makeop{Sym}   \makeop{can}
\makeop{length}\makeop{SO}    \makeop{torsion} \makeop{GSp} \makeop{Ker}
\makeop{Adm}
\def\makebb#1{\expandafter\def
  \csname bb#1\endcsname{{\mathbb{#1}}}\ignorespaces}
\def\makebf#1{\expandafter\def\csname bf#1\endcsname{{\bf
      #1}}\ignorespaces} 
\def\makegr#1{\expandafter\def
  \csname gr#1\endcsname{{\mathfrak{#1}}}\ignorespaces}
\def\makescr#1{\expandafter\def
  \csname scr#1\endcsname{{\EuScript{#1}}}\ignorespaces}
\def\makecal#1{\expandafter\def\csname cal#1\endcsname{{\mathcal
      #1}}\ignorespaces} 

\def\doLetters#1{#1A #1B #1C #1D #1E #1F #1G #1H #1I #1J #1K #1L #1M
                 #1N #1O #1P #1Q #1R #1S #1T #1U #1V #1W #1X #1Y #1Z}
\def\doletters#1{#1a #1b #1c #1d #1e #1f #1g #1h #1i #1j #1k #1l #1m
                 #1n #1o #1p #1q #1r #1s #1t #1u #1v #1w #1x #1y #1z}
\doLetters\makebb   \doLetters\makecal  \doLetters\makebf
\doLetters\makescr 
\doletters\makebf   \doLetters\makegr   \doletters\makegr
     \def\qed{\qedmark\medbreak}%
\def\qedmark{{\enspace\vrule height 6pt width 5pt depth 1.5pt}}%

\normalsize

\makeop{Bl}

\def\Fp{{\bbF}_p}

\def\Qbar{\overline{\bbQ}}

\newcommand{\Z}{\mathbb Z}
\newcommand{\Q}{\mathbb Q}

\newcommand{\F}{\mathbb F}

\newcommand{\npr}{\noindent }

\nc{\embed}{\hookrightarrow}

\newcommand{\ch}{characteristic }

\nc{\ol}{\overline}
\nc{\wt}{\widetilde}
\nc{\opp}{\mathrm{opp}}

\makeop{Ram}
\makeop{Rep}

\begin{document}
\renewcommand{\thefootnote}{\fnsymbol{footnote}}
\setcounter{footnote}{-1}
\numberwithin{equation}{section}


\title{Endomorphism algebras of QM abelian surfaces}
\author{Chia-Fu Yu}
\address{
Institute of Mathematics, Academia Sinica and NCTS (Taipei Office)\\
6th Floor, Astronomy Mathematics Building \\
No. 1, Roosevelt Rd. Sec. 4 \\ 
Taipei, Taiwan, 10617} 
\email{chiafu@math.sinica.edu.tw}


\date{\today.} 
\subjclass[2000]{11}              
                                         
\begin{abstract}
We determine endomorphism algebras of abelian surfaces
with quaternion multiplication. 
\end{abstract} 

\maketitle

\def\Mat{{\rm Mat}}
\def\c{{\rm c}}
\def\i{{\rm i}}

\section{Introduction}
\label{sec:01}

In this paper, we determine all possible endomorphism algebras
of abelian surfaces with quaternion multiplication (QM). 
Let $D$ be an indefinite quaternion division algebra 
over the field $\Q$ of rational
numbers. We would like to
find out all $\Q$-algebras $E$ containing $D$ which appear as endomorphism
algebras 
of abelian surfaces. 
In other words, we would like to know
which endomorphism algebra appears in the Shimura curve $X_D$
associated to the quaternion algebra $D$ (and with additional
data). Our main result states as follows.

\begin{thm}\label{11}
  Let $D$ be an indefinite quaternion division algebra over $\Q$, and
  let $A$ be an abelian surface over a field $k$ 
  with quaternion multiplication by $D$,
  i.e. an abelian surface together with a $\Q$-algebra 
  embedding $\iota: D\to E:=\End^0(A):=\End(A)\otimes_{\Z} \Q$. 
  \begin{itemize}
  \item [(1)] Suppose that $A$ is not simple. Then $A$ is isogenous to
    $C^2$ for an elliptic curve $C$ over $k$ and the algebra 
    $E$ is isomorphic to one of the
    following
    \begin{itemize}
    \item [(i)] $\Mat_2(K)$, where $K$ is any imaginary quadratic field
    which splits $D$,
    or
    \item [(ii)] $\Mat_2(D_{p,\infty})$, where $p$ is a prime and  
    $D_{p,\infty}$ is the
    quaternion algebra over $\Q$ ramified exactly at $\{p,\infty\}$. 
    This occurs if and only if $C$ is a
    supersingular elliptic curve over $k$ with $k\supset \F_{p^2}$. 
    \end{itemize}
    \item [(2)] Suppose that $A$ is simple. Then we have
   \begin{itemize}
  \item[(i)] $E\simeq D$, or
  \item[(ii)] $E\simeq D_K:=D\otimes_\Q K$ 
    for some imaginary quadratic field $K$. In this case, 
    the abelian surface $A$ is in \ch $p>0$ for some prime
    $p$ and it is supersingular. 
  \end{itemize}
  \end{itemize}
\end{thm}

Recall that an abelian
variety in \ch $p>0$ is said to be {\it supersingular} if it is
isogenous to a product of supersingular elliptic curves over a finite
field extension. 
The case (i) of Theorem~\ref{11} (2) occurs as we can take a
generic complex abelian surface with QM by $D$. 
The case (ii) of
Theorem~\ref{11} (2) occurs only when the quaternion algebra $D$
satisfies a special condition and the simple abelian variety $A$ is
necessarily supersingular. In this case, the algebra $E$ is
obviously determined by its center $K$, 
and we show that there are only a finite list of possibilities for
such $K$. More precisely, we have the following result.

\begin{thm}[Theorem~\ref{210}]\label{12}
  Let $A$ be a simple supersingular
  abelian surface over a finite field $\F_q$ of \ch $p>0$ with
  quaternion multiplication by $D$. Let $E:=\End^0(A)$ be the
  endomorphism algebra of $A$, and let $S$ be the
  discriminant of $D$. Then  
  \begin{itemize}
  \item [(1)] The center $K$ of $E$ is isomorphic 
    to $\Q(\zeta_n)$ for $n=3, 4$, or $6$.
  \item [(2)] One has $p\mid S$ and $p \equiv 1 \pmod n$, where $n$ is as
    above, and for any other prime $\ell \mid S$, one has either
    $\ell | n$ or $\ell \equiv -1 \pmod n$, that is, $\ell$ does not
    split in the quadratic field $\Q(\zeta_n)$.
  \item [(3)] $E\simeq D\otimes_\Q K$.   
  \end{itemize}  
\end{thm}

According to Theorem~\ref{12}, there are three possibilities for
endomorphism algebras $E$ of simple supersingular abelian surfaces
over finite fields: $E\simeq D\otimes_\Q \Q(\zeta_n)$ for $n=3,4, 6$. 
However, not all of them occur; It depends on the quaternion algebra $D$.
The algebra $D\otimes_\Q \Q(\zeta_n)$ occurs if and only if there is
exactly one prime $p$ dividing $S$ such that $p\equiv 1 \pmod n$ 
(This prime $p$ is the \ch of the base field). 

The results of this paper (Theorems~\ref{11} and \ref{12}) contribute
new cases to the problem about semi-simple algebras that can 
be realized as endomorphism algebras of abelian varieties. 
See Oort \cite{oort:endo} and the references therein for quite
complete discussions of this problem.  

Supersingular abelian surfaces appear in the classification of
endomorphism algebras of QM abelian surfaces. We refer the reader to
C.~Xing \cite{xing:ss1996} for some aspects of supersingular abelian
surfaces over finite fields. 
   

Let $J$ be the Jacobian of a smooth, projective, geometrically
connected algebraic curve of genus $2$ over a number field $K$. In
\cite{baba-granath:QM} Baba and Granath showed that if the following
three 
conditions hold: (1) $J$ has QM by a maximal quaternion order
$\Lambda_6$ of discriminant $6$, (2) $J$ has the 
field of moduli equal to $\Q$, and (3)
$J$ has potentially smooth stable reduction at both $2$ and $3$, 
then QM abelian surface $J$ has superspecial good reduction at 
infinitely many primes. In \cite{dieulefait-rotger:QM} Dieulefait and 
Rotger studied
the arithmetic of the Jacobians $J$ whose endomorphism algebra 
$\End^0_{\Qbar}(J)$ is an indefinite quaternion algebra. 
They determine all possible Galois groups of minimal
fields of definition and possible endomorphism rings $\End_{K}(J)$
defined over a smaller number field $K$ under a certain integral
condition. We refer the
reader to \cite{dieulefait-rotger:QM} for the list of Galois groups
and more detail discussions. 

An analogous question to our main results (Theorems~\ref{11} and
\ref{12}) is: What are the endomorphism algebras of abelian
varieties with real multiplication (RM)? That is, one considers 
the same problem as treated in this paper but for Hilbert modular 
varieties rather than Shimura curves.  
This problem has been done by 
Chai \cite[Section 3]{chai:ho}.  
The classification has its own interest; this also plays a role in 
the proof of Chai's theorem on the density of ordinary Hecke 
orbits in Siegel modular varieties. 
As the reader may be also interested in this result 
due to Chai, we include an expository and elementary account for the 
reader's convenience in Section~\ref{sec:03}. Using the similar method 
as in Section~\ref{sec:02}, 
we make Chai's result more explicit about the simple algebras 
that actually occur as endomorphism algebras of RM abelian
varieties.

\section{Proof of Main Results}
\label{sec:02}

\subsection{Embeddings of simple algebras}
\label{sec:21}
We recall some basic definitions for central simple algebras; 
see \cite{reiner:mo}.

\begin{defn}
  Let $B$ be a (f.d.) central simple algebra over a field
  $F$. The {\it degree}, {\it capacity}, and {\it index} of $B$
are defined as
\[ \deg(B):=\sqrt{[B:F]},\quad \c(B):=n,\quad
  \i(B):=\sqrt{[\Delta:F]}, \] 
respectively,  
if $B\cong \Mat_n(\Delta)$, where $\Delta$ is a division algebra over
$F$, which is uniquely determined by $B$ up to isomorphism.
The algebra $\Delta$ is also called the {\it division part} of $B$.
\end{defn}

\begin{prop}\label{22}
  Let $E$ and $B$ be two finite-dimensional simple algebras over a
  field $F$ with centers $Z$ and $K$, respectively. Suppose
  that $Z$ and $K$ are linearly disjoint over $F$, that is, the
  $F$-algebra $L:=Z\otimes_F K$ is a field. Let $E\simeq
  \Mat_n(\Delta)$, where $\Delta$ is the division part of $E$. Then
  there is an $F$-algebra embedding of $B$ into $E$ if and only if 
  \begin{equation}
    \label{eq:11}
    [B:F]\mid n \cdot c, 
  \end{equation}
  where $c$ is the capacity of the central simple algebra
  $\Delta\otimes_F B^{\rm o}$ over $L$:
\[ \Delta\otimes_F B^{\rm o}\simeq 
  \Delta\otimes_Z (Z\otimes_F K) \otimes_K B^{\rm o}\simeq 
  (\Delta\otimes_Z L)\otimes_L (L\otimes_K
  B^{\rm o}), \]
and $B^{\rm o}$ denotes the opposite algebra of $B$.
\end{prop}

\begin{proof}
  This is a special case of \cite[Theorem 1.2]{yu:embed}. However, 
  instead of referring to the general result, 
  we prefer to give a direct proof for the
  reader's convenience. 
  Let $E=\End_{\Delta}(V)$, where $V$ is a right vector space over
  $\Delta$. An $F$-algebra embedding from
  $B$ into $E$ exists if and only if $V$ is a $(B,\Delta)$-bimodule,
  or equivalently
  a right $\Delta\otimes_F B^{\rm o}$-module. Let $\Delta\otimes_F
  B^{\rm o} \simeq \Mat_c(\Delta')$, where $\Delta'$ is the
  division part of the simple algebra $\Delta\otimes_F B^{\rm o}$.
  By the dimension
  counting, the vector space $V$ is a $\Mat_c(\Delta')$-module 
  if and only if  
  \begin{equation}
    \label{eq:22}
    \frac{\dim_F V}{c  [\Delta':F]}\in \bbN.
  \end{equation}
  Note that
  $[B:F][\Delta:F]=c^2[\Delta':F]$. 
  From this relation and that $\dim_F V=n[\Delta:F]$,  
  the condition (\ref{eq:22}) can be written as $[B:F]\mid nc$. 
  This proves the proposition. \qed 
\end{proof}

\begin{remark}\label{23}
  The reader can find in \cite{yu:embed} for more general results about  
  Proposition~\ref{22} where $B$ and $E$ are any finite-dimensional
  semi-simple $F$-algebras. 
  When $F$ is a global field, the local-global principle enters and
  plays a role in the problem of embeddings of simple algebras. For
  a detailed discussion, the reader is referred to 
  the paper \cite{shih-yang-yu}.

 
\end{remark}

After establishing a basic embedding result (Proposition~\ref{22}), 
we begin with the
classification of endomorphism algebras of QM abelian surfaces. 
Let $(A,\iota)$ be an abelian surface with quaternion multiplication
by $D$ and let $E:=\End^0(A):=\End(A)\otimes_{\Z} \Q$ 
be the endomorphism algebra of $A$. \\ 

\subsection{Case where $A$ is not simple.}
\label{sec:22}
In this case $A$ is isogenous to $C_1\times C_2$, 
where $C_1$ and $C_2$ are elliptic curves. Then $C_1$ is
isogenous to $C_2$. If not, then we have 
inclusions $D \subset \End^0(C_i$) for $i=1,2$ and each $C_i$ must be
supersingular. It follows that $D\simeq \End^0(C_i)\simeq
D_{p,\infty}$, the definite quaternion algebra over $\Q$ 
ramified exactly at $\{p,\infty\}$, contradiction. 

Therefore, the algebra $E=\End^0(A)=\End(A)\otimes \Q$ is isomorphic to
one the following:
\begin{itemize}
\item[(i)] $\Mat_2(\Q)$,
\item [(ii)] $\Mat_2(K)$, where $K$ is an imaginary quadratic field, 
\item [(iii)] $\Mat_2(D_{p,\infty})$. 
\end{itemize}

The case (i) can not occur because $D$ and $\Mat_2(\Q)$ are 
different quaternion algebras. 

The case (ii) can appear if and
only if $K$ splits $D$. Indeed, as one has an embedding of $D$ in
$\Mat_2(K)$, the algebra 
$D$ acts on a $K$-vector space $V$ of
dimension two. We can identity $V$ with $D$ as $V$ is one-dimensional
$D$-vector space. This makes $D$ a $K$-vector
space of dimension $2$. 
Therefore, $K$ is isomorphic to a (necessarily) maximal 
subfield of $D$. This is exactly when $K$ splits $D$. 
On the other hand, 
any imaginary quadratic field is isomorphic to 
the endomorphism algebra of an 
elliptic curve. Therefore, for any imaginary quadratic field that
splits $D$, the matrix algebra 
$\Mat_2(K)$ can occur as the  endomorphism algebra of a QM
abelian surface.  

For the case (iii), this occurs of course only when $A$ is in
\ch $p>0$ and $A$ is isogenous to the product of two supersingular
elliptic curves over the base field $k$ containing $\F_{p^2}$. 
Now we check that an embedding 
$\iota:D\to \Mat_2(D_{p,\infty})$ exists for any prime $p$. 
By Proposition~\ref{22}, we
need to show that $[D:\Q] \mid 2 c$, where $c$ is the capacity of the
central simple algebra $D^{\rm o}\otimes_\Q D_{p,\infty}$ over $\Q$ 
and $D^{\rm o}$ denotes the opposite algebra of $D$. 
As the tensor product of two quaternion algebras is Brauer equivalent
to a quaternion algebra, we have
\[ D^{\rm o}\otimes_\Q D_{p,\infty}\simeq \Mat_2(D') \] 
for some definite quaternion algebra $D'$ over $\Q$. This shows that
$c=2$ and hence that a $\Q$-algebra embedding 
$\iota: D\to \Mat_2(D_{p, \infty})$ exists. 
We have shown the following result.

\begin{prop}\label{24}
  Let $D$ be as above and $A$ be an abelian surface with
  endomorphism algebra $E:=\End^0(A)$ containing $D$. 
  Suppose that $A$ is not simple. Then
  $A$ is isogenous to $C^2$ for an elliptic curve $C$ 
  and the algebra $E$ is
  isomorphic to one of the following two cases 
  \begin{itemize}
  \item [(i)] $\Mat_2(K)$, where $K$ is any imaginary quadratic field
    which splits $D$,
    or
  \item [(ii)] $\Mat_2(D_{p,\infty})$. This occurs if and only 
    if $C$ is a supersingular elliptic curve in \ch $p>0$ and the base
    field $k$ contains $\F_{p^2}$. 
  \end{itemize}
\end{prop}


\subsection{Case where $A$ is simple.}
\label{sec:23}

Since $E:=\End^0(A)$ 
contains the quaternion algebra $D$, the algebra $E$ is
non-commutative. Let $K$ be the center of $E$. Since $\dim A=2$, any
maximal subfield of $E$ has degree $2$ or $4$ (over $\Q$). So one has
$[K:\Q]|4$. If
$[K:\Q]=4$, then $E=K$ (which is commutative), absurd. 
So $K=\Q$ or $[K:\Q]=2$. 

If $K=\Q$,
which is totally real,
then $E$ is a quaternion algebra over $\Q$. This follows from 
Albert's classification of central division algebras with positive
involution (cf. Mumford \cite[Section 21]{mumford:av}).  
In this case, one must have $E\simeq D$. 

Suppose now that $[K:\Q]=2$. Then $E$
is a quaternion division algebra over $K$. If $K$ is real, then
$E\supset D\otimes_\Q K$
contains a totally real maximal subfield $K'$ (of degree $4$ over $\Q$), 
which shows that $\dim A$ is divisible by $4$ (see Mumford
\cite[Corollary, p.~191]{mumford:av}), 
absurd. It follows that the center $K$ is an
imaginary quadratic field. Note that in this case, 
$A$ is in \ch $p>0$ for some prime $p$.
Indeed, its endomorphism algebra contains a $4$-dimensional CM subfield
We also know that 
any simple CM abelian variety by a CM field $L$ in \ch zero has
endomorphism algebra equal to $L$ but $E$ is non-commutative. 
Therefore, $A$ is in positive characteristic. 
Now we determine which
quaternion division algebra $E$ over an imaginary quadratic field $K$
contains a subalgebra isomorphic to $D$. 
This is exactly when the capacity of $E\otimes_\Q D^{\rm o}$ is equal
to $4$ by Proposition~\ref{22}, or equivalently, 
the quaternion algebra $D_K:=D\otimes_\Q K$ is 
isomorphic to $E$. We have shown the following result.

\begin{prop}\label{25}
  Let $D$ be as above and $A$ be an abelian surface with quaternion
  multiplication by $D$.
  Suppose that $A$ is simple. Then
  \begin{itemize}
  \item[(i)] $E\simeq D$, or
  \item[(ii)] $E\simeq D_K:=D\otimes_\Q K$ 
    for some imaginary quadratic field $K$. In
    this case, the abelian surface $A$ is in \ch $p>0$ for some prime
    $p$. 
  \end{itemize}
\end{prop}

Theorem~\ref{11} follows from Propositions~\ref{24} and \ref{25}.

The case (i) of Proposition~\ref{25} 
occurs as one can take $A$ to be a generic complex
QM abelian surface. For the case (ii), we make a further discussion 
about the algebras of the form $D_K$ that can occur in the next subsection.

\subsection{}
\label{sec:24}
Put $D_K:=D\otimes_\Q K$, where $K$ is an imaginary quadratic field.
In the remaining of this section, we investigate which $D_K$ can be
realized as the endomorphism algebra of an abelian surface.   

Suppose $D_K\simeq \End^0(A)$ for an abelian surface. Then $A$ has smCM
(sufficiently many complex multiplications, that is, the endomorphism
algebra $\End^0(A)$ of $A$ contains a semi-simple commutative
$\Q$-subalgebra $L$ with $[L:\Q]=2\dim A$). By a
theorem of Grothendieck \cite[Section 22, p.~220]{mumford:av}, 
there are a finite field extension $k'$ of the
ground field $k$ and an abelian surface $A_0$ over a finite field
$k_0$ contained in $k'$ such that there is an isogeny $A\otimes_k
k'\to A_0\otimes_{k_0} k'$ over $k'$ 
(see \cite{oort:cm} for Grothendieck's
original proof and \cite{yu:cm} for a different proof). 
We may enlarge $k_0$ in $k'$, 
if necessary, such that 
$\End^0(A_0\otimes_{k_0} k')=\End^0_{k_0}(A_0)=:E_0$. 
This shows the following:

\begin{lemma}
  Notations being as above, the algebra $E\simeq D_K$ is 
  contained in the endomorphism algebra $E_0$ of an abelian surface 
  over a finite field. 
\end{lemma}

Since $[E:\Q]=8$,
one has either 
\begin{itemize}
\item [{\bf (a)}] $E=E_0$, or
\item [{\bf (b)}] $\dim E_0=16$. 
\end{itemize}

\begin{lemma}\label{26}
Let notations be as above. For either the case {\bf (a)} or {\bf (b)}, 
there is a  rational prime $p$ which splits in $K$. 
Furthermore we have for any finite
place $v$ of $K$ 
\begin{equation}
  \label{eq:23}
  \inv_v(E)=
\begin{cases}
  1/2 & \text{if $v|p$,} \\
  0 & \text{otherwise}. 
\end{cases}
\end{equation}
\end{lemma} 
\begin{proof}
  For the case {\bf (b)}, the algebra $E_0\simeq \Mat_2(D_{p,\infty})$
  for 
  some rational prime $p$. The
  algebra $D_K$ can be embedded in $E_0$ if and only if
  $D_{p,\infty}\otimes_\Q  K\simeq D_K\simeq E$. Since
  $\inv_\infty(E)=0$ and 
  $E$ is a division algebra, the places with non-trivial invariants
  are those of $K$ over $p$. It follows that there are two places of
  $K$ lying over $p$ at which $E$ has non-trivial local invariant, 
  and the remaining local invariants are trivial. 

  For the case {\bf (a)}, $E$ is the endomorphism algebra of an abelian
  surface $A_0$ over a finite field $k_0$. 
  Using the Honda-Tate theory, the center
  $K$ is $\Q(\pi_0)$, where $\pi_0$ is the relative Frobenius
  endomorphism 
  of $A_0$ over $k_0$. For any finite place $v$ of $K$ with $v \nmid
  p$, one has $\inv_v(E)=0$. As $E$ is a division algebra, it follows
  that there are two places of
  $K$ lying over $p$ at which $E$ has non-trivial local invariant, 
  and the remaining local invariants are trivial. \qed   
\end{proof}

We need to find all rational primes $p$ and imaginary quadratic fields
$K$ such that the quaternion algebra $D_K:=D\otimes_\Q K\simeq E$ 
satisfies the condition of
Lemma~\ref{26} and that the algebra
$E$ appears as the endomorphism algebra of an abelian surface. 
Let $S$ be the discriminant of $D$ over $\Q$; by definition, $S$ is the
product of all finite ramified rational primes for $D$. Clearly, one has
$p\mid S$, otherwise the local invariants of $D_K$ at places $v$ lying
over $p$ are zero and hence that $D_K\simeq \Mat_2(K)$, absurd. 
Therefore, a necessary condition that $K$ satisfies the conditions in
Lemma~\ref{26} is the following: \\

($*$) The prime $p$ splits in $K$ and for any other prime $\ell | S$, the
completion $K_\ell:=K\otimes \Q_\ell$ at $\ell$ is a field. \\

The first condition of ($*$) follows from $\inv_v(E)=1/2$ if $v|p$ and the
second one follows from $\inv_v(E)=0$ otherwise. 

Now given a rational prime $p\mid S$ and an imaginary quadratic field
$K$ satisfying the condition ($*$), we would like to  find a 
Weil $q$-number $\pi$, where $q$
is a power of $p$,  so that
$K\simeq \Q(\pi)$ and for every place $v \mid p$ of $K$, one has 
$v(\pi)/v(q)=1/2$.  

\begin{lemma}\label{27}
  Let $(A,\iota)$ be an abelian surface with QM by $D$ 
over a field $k$
  of \ch $p>0$. Suppose that $p \mid S$, then $A$ is supersingular. 
\end{lemma}
\begin{proof}
  This result is well-known (see \cite{boutot-carayol:sc}); 
  we provide a proof for the reader's
  convenience. We may assume that the ground 
  field $k$ is algebraically closed. We have
  a $\Q_p$-algebra embedding $\iota:D_p:= D\otimes \Q_p\to \End^0(X)$,
  where $X:=A[p^\infty]$ is the $p$-divisible group attached to
  $A$. The possibilities of $\End^0(X)$ are (a) (ordinary)
  $\Mat_2(\Q_p)\times \Mat_2(\Q_p)$, (b) ($p$-rank one) $D_p\times
  \Q_p\times \Q_p$, and (c) (supersingular) $\Mat_2(D_p)$. Clearly,
  only the case (c) is possible. Therefore, the abelian surface $A$ is
  supersingular. \qed    
\end{proof}

We need to find all Weil $q$-numbers $\pi$ so that the corresponding
abelian variety $A_\pi$, uniquely determined up to isogeny, is 
both simple and supersingular, 
and that its center $\Q(\pi)$ is an imaginary quadratic
field satisfying the condition ($*$). The latter condition will
imply that the endomorphism algebra of $A_\pi$ is a quaternion
division algebra over $K$ and hence that 
$A_\pi$ is an abelian surface.

\begin{thm}\label{28}
  Let $q$ be a power of a prime number $p$ and $\pi$ is a Weil
  $q$-number. Then the corresponding abelian variety $A_\pi$ 
  is supersingular if
  and only if $\pi=\sqrt{q}\zeta$, where $\zeta$ is a root of unity. 
\end{thm}
\begin{proof}
  This is a well-known immediate consequence 
  of results due to Manin \cite{manin:thesis}, Tate \cite{tate:ht} and 
  Oort \cite[Theorem 2]{oort:product}, also see \cite{xing:ss1996}. We
  provide a proof for the 
  reader's convenience. Let $C$ be 
  a supersingular elliptic curve over $\Fp$ such that
  $\pi_{C}^2+p=0$, where $\pi_{C}$ is the Frobenius endomorphism
  of $C/\Fp$. Put $A_1=C^g$, where $g=\dim A_\pi$. Since any two
  supersingular abelian varieties are isogenous over a finite
  extension of their ground fields, we have $\pi^N=p^M$ for some
  positive integers $N$ and $M$. It follows that $\pi$ is of the
  form $\sqrt{q} \zeta$, where $\zeta$ is a root of unity. 
  Conversely, suppose that $\pi$ is of this form. Then 
  $\pi^N$, for some even integer $N$, 
  is $q^{N/2}$,  which is a Weil number corresponding 
  to a supersingular elliptic curve. By Tate's isogeny theorem, 
  $A$ is isogenous to the product of copies
  of a supersingular elliptic curve over a finite field. This
  completes the proof of the theorem. \qed   
\end{proof}

We shall call a Weil $q$-number $\pi$ {\it supersingular } if the
corresponding simple abelian variety $A_\pi$ up to isogeny is
supersingular. 

\begin{lemma}\label{29}
  Let $\pi=\sqrt{q}\,\zeta_n$ be a supersingular Weil $q$-number, where
  $q=p^a$, and
  $\zeta_n$ is a primitive $n$-th root of unity. Then the field
  $K=\Q(\pi)$ generated by $\pi$ is an
  imaginary quadratic field if and only if 
\item [(a)] $a$ is even and $n=3,4,6$, or
\item [(b)] $a$ is odd and $n=4$, $n=8$ and $p=2$, or $n=12$ and $p=3$.
\end{lemma}
\begin{proof}
  (a) One has $\Q(\pi)=\Q(\zeta_n)$, so $[\Q(\pi):\Q]=2$ if and only
  if $n=3,4, 6$. (b) We have $\Q(\pi)=\Q(\sqrt{p}\zeta_n)$. Since
  $\Q(\pi^2)=\Q(\zeta_m)$ has degree one or two, $m=2,3,4$ or $6$ and
  $n=4, 3, 6, 8$ or $12$. The case $n=4$ is good. For the remaining
  cases one must have $\Q(\pi)=\Q(\pi^2)$. Note that $2$ is the
  only ramified prime in $\Q(\zeta_4)$ and $3$ is the only ramified
  prime in $\Q(\zeta_3)=\Q(\zeta_6)$. Since $p$ is ramified in
  $\Q(\pi)$, the only possibilities are $\Q(\sqrt{2}\,\zeta_8)$ and
  $\Q(\sqrt{3} \zeta_n)$ for $n=3,6,12$. It is easy to check that only
  $\Q(\sqrt{2}\,\zeta_8)$ and $\Q(\sqrt{3}\, \zeta_{12})$ are quadratic
  fields. \qed

\end{proof}

Note that in the case (a) the prime $p$ splits in $\Q(\pi)$ if and
only if $p\equiv 1 \pmod n$. In the case (b) $p$ is ramified in
$\Q(\pi)$. Therefore, if one requires the field $\Q(\pi)$ 
satisfy the condition ($*$), then only the case (a) can occur. 
This yields the following conclusion.

\begin{thm}\label{210} 
  Let $A$ be a simple supersingular
  abelian surface over a finite field $\F_q$ of \ch $p>0$ with
  quaternion multiplication by $D$ and let $E:=\End^0(A)$. Then  
  \begin{itemize}
  \item [(1)] The center $K$ of $E$ is isomorphic 
    to $\Q(\zeta_n)$ for $n=3, 4$, or $6$.
  \item [(2)] One has $p\mid S$ and $p \equiv 1 \pmod n$, where $n$ is as
    above, and for any other prime $\ell \mid S$, one has either
    $\ell | n$ or $\ell \equiv -1 \pmod n$.
  \item [(3)] $E\simeq D\otimes_\Q K$.   
  \end{itemize}  
\end{thm}

It follows from Theorem~\ref{210} that there are three possibilities
for 
endomorphism algebras $E$ of simple supersingular abelian surfaces
over finite fields: $E\simeq D\otimes_\Q \Q(\zeta_n)$ for $n=3,4, 6$. 
Furthermore, the algebra $D\otimes_\Q \Q(\zeta_n)$ occurs if and only
if there is 
exactly one prime $p\mid S$ such that $p\equiv 1 \pmod n$. This prime
is the \ch of the ground field of the abelian surface.

We conclude this section with a remark and a question we think
interesting. Let $E$ be the endomorphism algebra of a
simple QM abelian surface $A$ such that $E\neq D$. Then
\begin{itemize}
\item $A$ is in \ch $p>0$ for a prime $p$, and $A$ is supersingular, 
\item $E=D\otimes_\Q K$, where $K$ is an imaginary quadratic field
  satisfying the condition ($*$) after Lemma~\ref{26}. 
\item In case that $E$ is isomorphic to the endomorphism algebra of
  supersingular simple QM abelian surface over a finite field, 
  we know all such $K$ that can occur by Theorem~\ref{210}.  
\end{itemize}

However, we are not able to rule out the possibility that $E$ is not
isomorphic to a (necessarily supersingular) simple QM abelian surface
over {\it a finite field}. That is, whether or not the endomorphism
algebra of any supersingular abelian surface over an 
{\it arbitrary field} $k$ is isomorphic to 
that of one over a finite field.
We make the following hypothesis.  \\ 

\npr {\bf (H)} Let $k$, $k'$ and $k_0$ be three fields with the
inclusion relation $k\subset k'\supset k_0$. 
Let $A/k$ and $A_0/k_0$ be two abelian varieties such that there
is an isogeny $\varphi: A\otimes_k k' \to A_0 \otimes_{k_0}
k'$ over $k'$.   
Suppose that $\End^0(A_0)=\End^0(A_0\otimes_{k_0} k')$. We
identify $\End^0(A \otimes_k k')=\End^0(A_0)$ using the isogeny 
$\varphi$. Then there an abelian variety $A_1/k_1$ over a subfield
$k_1\subset k_0$ and an isogeny 
$\varphi_1: A_1\otimes_{k_1} k_0 \to A_0$ over
$k_0$ so that the subalgebra $\End^0(A_1)\subset \End^0(A_0)$ (through
$\varphi_1$) is
equal to the subalgebra 
$\End^0(A)$ in  $\End^0(A \otimes_{k} k')=\End^0(A_0)$.      \\

The hypothesis {\bf (H)} rules out the
possibility of imaginary quadratic fields $K$ satisfying the necessary
condition ($*$) that are not of the shape described in
Theorem~\ref{210}.
Then Theorems~\ref{11} and \ref{12} give a complete result for
endomorphism algebras of QM abelian surfaces over an arbitrary base
field. Besides, we also obtain the following result, which is 
an immediate consequence of a theorem of
Grothendieck (cf. \cite{oort:cm}, \cite{yu:cm}). 

\begin{cor}\label{211}
  Let $A/k$ be an abelian variety that has smCM 
  over a field $k$ of \ch $p>0$. 
  Assume the 
  hypothesis {\bf (H)}. Then the endomorphism
  algebra $\End^0(A)$ of $A/k$ is isomorphic to that of an abelian
  variety over  a finite field. 
\end{cor}

Even when the hypothesis {\bf (H)} fails, one is still
in an interesting situation.
This means there are some subtle issues about the fields of definition
that we were not aware of. For example, there are endomorphism
algebras of abelian
varieties having smCM in positive \ch that can not be found by the
Honda-Tate theory. These would contribute new examples to the problem
of endomorphism algebras of abelian varieties studied in Oort 
\cite{oort:endo}.      
 

The following question should be helpful to understand the problem of
fields of definition arising from {\bf (H)}. Recall that an abelian
variety over a field $k$ of \ch $p>0$ is said to be {\it superspecial}
if it is isomorphic to a product of supersingular elliptic curves over
an algebraic closure of $k$. \\

\npr {\bf (Q)}. Let $A$ be a superspecial abelian variety over a field
$k$ of \ch $p>0$. Is there a superspecial abelian variety $A_0$ over a
finite field $k_0$ so that $A$ is isomorphic to $A_0\otimes_{k_0} k$
over $k$?   

\section{Endomorphism algebras of RM abelian varieties}
\label{sec:03}

In this section, we give an exposition on endomorphism algebras of
abelian varieties with real multiplication. Our reference is Chai
\cite{chai:ho}, especially Section 3 of it. 
The classification has its own interest; this is 
also useful in the proof of Chai's theorem on the density 
of ordinary Hecke orbits in Siegel modular varieties.
We change the notations a bit. 
Let $F$ be a totally real number field of degree $g=[F:\Q]$, and
let $O_F$ be the ring of integers. {\it An abelian variety with real
multiplication by $O_F$} is a pair $(A,\iota)$, where $A$ is a
$g$-dimensional abelian variety and $\iota: O_F\to \End(A)$ is a ring
monomorphism. As we are only concerned with the endomorphism algebra
$\End^0(A)$ of such objects, we may replace 
the ring monomorphism 
$\iota: O_F\to \End(A)$ by its induced $\Q$-algebra embedding 
$\iota: F\to \End^0(A)$. We shall call the latter object $(A,\iota)$ 
{\it an abelian variety with RM by $F$}.  
 
Let $(A,\iota)$ be an abelian variety with RM by $F$ over an
(unspecified) base field $k$.

\begin{lemma}
  The underlying abelian variety $A$ is isogenous to $A_1^n$, where $A_1$
  is a simple abelian variety.
\end{lemma}
\begin{proof}
  Let $A$ be isogenous to $\prod_{j=1}^r A_{j}^{n_j}$, where $A_i$ and
  $A_j$ are non-isogenous simple abelian varieties if $i\neq j$. Then
  one has a $\Q$-algebra embedding $\iota:F \to \End^0(A_j^{n_j})$ for
  each $j=1,\dots, r$, and hence $g \mid \dim A_j^{n_j}$. It follows
  that $r=1$. \qed  
\end{proof}

Put $d:=\dim A_1$ and $\Delta:=\End^0(A_1)$. One has $g=nd$ and
$\End^0(A)=\Mat_n(\Delta)$. The division algebra $\Delta$ admits a
positive (Rosati) involution $*$. We use the classification of
Albert for $\Delta$ (cf. \cite[Section 21]{mumford:av}).\\

\npr (Type I) The algebra $\Delta=K_0$ is a totally real 
number field. Since
$F\simeq \iota(F)\subset \Mat_n(K_0)$, 
one has $g\mid n [K_0:\Q]$. On the other hand,
one has $[K_0:\Q]\mid d$. It follows that $d=[K_0:\Q]$. The map
$\iota$ makes the $K_0$-vector space $V=K_0^n$ as an $F$-vector space
of dimension one. Therefore, $F$ can be regraded as a $K_0$-vector
space and hence $K_0$ is isomorphic to a subfield of $F$. We may
assume that $K_0$ is a subfield of $F$. Conversely, given a subfield
$K_0$ of $F$ of degree $d$, we choose an abelian variety $A_1$ so that
the endomorphism algebra $\End^0(A_1)$ is isomorphic to $K_0$. Take
$A:=A_1^n$, where $n:=g/d$. Since there is a $\Q$-algebra
embedding $\iota: F\to \Mat_n(K_0)$, we have an abelian variety
$(A,\iota)$ with RM by $F$ such that $\End^0(A)$ is isomorphic to
$\Mat_n(K_0)$. \\

 \npr (Type II) The algebra $\Delta$ is a totally indefinite
 quaternion algebra over a totally real number field $K_0$. Since
 $\Delta$ contains a totally real maximal subfield $K_1$, which has
 degree $2[K_0:\Q]$, one has $2 [K_0:\Q]\mid \dim A_1$. On the other
 hand, any maximal subfield of $\Mat_n(\Delta)$ has degree $2 n
 [K_0:\Q]$, which gives the other divisibility $g\mid 2n
 [K_0:\Q]$. Therefore, we have $g=2 n [K_0:\Q]$ and $\dim A_1=2
 [K_0:\Q]$. Since $F$ has the degree of maximal
 semi-simple commutative subalgebras of $\Mat_n(\Delta)$, the image of
 any embedding $\iota$ contains the center $K_0$. Therefore, we may
 assume that $K_0$ is a subfield of $F$ and $\iota$ is a $K_0$-algebra
 embedding of $F$ into $\Mat_n(\Delta)$. The field $F$ is isomorphic
 to a (maximal) subfield of $\Mat_n(\Delta)$ if and only if $F$ splits
 $\Delta$. The latter is equivalent to that for any place 
$v\in {\rm Ram}(\Delta/K_0)$, the set of ramified places of $K_0$ for
 $\Delta$, and any place $w|v$ of $F$, one has $[F_w:K_{0,v}]\equiv 0
 \pmod 2$. 
  
Conversely, suppose we have a subfield $K_0$ of $F$ with
$2[K_0:\Q]n=g$ for some positive integer $n$, 
and a totally indefinite quaternion algebra $\Delta$ over $K_0$ which
is split by $F$. Then there exists an abelian variety with RM by $F$
so 
that $\End^0(A)\simeq \Mat_n(\Delta)$. Indeed, we first take a complex
abelian variety $A_1$ with dimension $2[K_0:\Q]$ so that
$\End^0(A_1)\simeq \Delta$. Then put $A:=A_1^n$. The condition that
$F$ splits $\Delta$ implies that there is a $K_0$-algebra embedding
from $F$ into $\Mat_n(\Delta)$. This way we construct a (complex)
abelian   variety with RM by $F$ 
so that $\End^0(A)\simeq \Mat_n(\Delta)$. \\

 \npr (Type III) The algebra $\Delta$ is a totally definite
 quaternion algebra over a totally real number field $K_0$. Since $F$
 can be embedded in $\Mat_n(\Delta)$, whose maximal semi-simple commutative
 subalgebras have the same degree $2n[K_0:\Q]$, one has
 $g|2n[K_0:\Q]$. On the other hand, as the field $K_0$ acts on the
 abelian variety $A_1$ up to isogeny, one has $[K_0:\Q]|\dim A_1$;
 this gives the condition $n[K_0:\Q]|g$. Also, if the ground field $k$
 is of \ch zero, then the algebra $\Delta$ acts on the homology group
 $H_1(A_1,\Q)$. This gives the condition $4[K_0:\Q]|2\dim A_1$, or
 equivalently $2n[K_0:\Q]|g$. We have two cases:
 \begin{itemize}
 \item [(a)] $g=n[K_0:\Q]$. This case occurs only when $k$ is of \ch
   $p>0$ for some prime $p$. 
 \item [(b)] $g=2n[K_0:\Q]$.
 \end{itemize}
We first rule out the possibility of (b). Since $g$ is the degree of any
maximal semi-simple commutative subalgebra of 
$\Mat_n(\Delta)$, the image of $F$ under any embedding 
$\iota:F\to \Mat_n(\Delta)$ contains the
center $K_0$. Therefore, $F$ contains a subfield which is isomorphic
to $K_0$, 
and we may assume that the field $F$ contains $K_0$ and the embedding
$\iota$ is a $K_0$-algebra homomorphism. Since $F$ has the degree 
of $\Mat_n(\Delta)$ over $K_0$, the field $F$ can be embedded into
the simple algebra $\Mat_n(\Delta)$ if and only if $F$ splits
$\Delta$. But the latter is impossible 
because $\Delta$ is totally definite and $F$ is totally real. 

For the case (a), we have $\dim A_1=[K_0:\Q]$. In this case 
the abelian variety $A_1$ has
smCM. By a theorem of Grothendieck \cite[Section 22, p. 220]{mumford:av},
there are a finite field
extension $k'/k$, an abelian variety $A_0$ over a finite field $k_0$
contained in $k'$ and an isogeny $\varphi: A_1\otimes_k k' \to
A_0\otimes_{k_0} k'$. Enlarging $k_0$ if necessary, we can assume that
$\End^0(A_0\otimes_{k_0} k')=\End^0(A_0)$. We have
\[ \Delta=\End^0(A_1)\subset \End^0(A_1\otimes_{k} k')\simeq
\End^0(A_0\otimes_{k_0} k')=\End^0(A_0). \]  
We first show that $A_0$ (and $A_1$) is supersingular. We know that
both $\Delta$ and $\End^0(A_0)$ have the same degree (=$2\dim A_1$) 
of maximal semi-simple commutative subalgebras. The centralizer of
$K_0$ of $\End^0(A_0)$ is equal to the division algebra $\Delta$ 
and hence by bi-commutant theorem \cite[Theorem 7.11 and Corollary
7.13, p.~94-95]{reiner:mo} that the centralizer 
of $\Delta$ in $\End^0(A_0)$ is equal to
$K_0$. It follows that the center $Z$ of $\End^0(A_0)$ is contained in
the totally real field $K_0$. By the classification of endomorphism
algebras of abelian varieties over finite fields in Tate
\cite{tate:eav}, the field $Z$ is equal to $\Q$ or $\Q(\sqrt{p})$ and
$A_0$ is supersingular. We may enlarge the field $k_0$ so that $A_0$
is isogenous to the product of copies of supersingular elliptic curves
with endomorphism algebra $D_{p,\infty}$. Therefore,
$\End^0(A_0)=\Mat_{[K_0:\Q]}(D_{p,\infty})$, noting that $\dim
A_0=\dim A_1=g/n=[K_0:\Q]$. Then the division algebra
$\Delta$ is equal to the centralizer of $K_0$ in
$\Mat_{[K_0:\Q]}(D_{p,\infty})$. It follows that $\Delta$ is ramified
exactly at all Archimedean places and finite places $v$ of $K_0$ over
$p$ of odd degree, or equivalently $\Delta\simeq
D_{p,\infty}\otimes_{\Q} K_0$. 
If we assume  the hypothesis {\bf (H)} (at the end of
Section~\ref{sec:02}), then there are only two
possibilities for $\Delta$: either $\Delta=D_{p,\infty}$ or 
$\Delta=D_{\infty_1, \infty_2}$, the
definite quaternion algebra over the field $\Q(\sqrt{p})$ which is
ramified exactly at two Archimedean places $\infty_1$ and $\infty_2$. 
We have given all possibilities of the division algebras $\Delta$, 
and only the cases
$\Delta=D_{p,\infty}$ and $\Delta=D_{\infty_1,\infty_2}$ can occur as
endomorphism algebras of simple abelian varieties over finite fields. 

We now show that the field $F$ contains a subfield which is isomorphic
to $K_0$, so that we may assume that $F$ contains $K_0$ and that the
embedding $\iota$ is a $K_0$-algebra homomorphism. Let $x\mapsto
\bar{x} $ be 
the canonical involution of $\Delta$, which is the unique positive
involution. Define a positive involution $*$ on $\Mat_n(\Delta)$ by
$(a_{ij})^*=(\ol {a_{ji}}\,)$. We know that for any embedding
$\iota:F\to \Mat_n(\Delta)$, there is a positive involution $*_1$
which leaves the image $\iota(F)$ invariant and every element of
$\iota(F)$ invariant. On the other hand, one can show that there is an
isomorphism of algebras with involution $(\Mat_n(\Delta),*_1)\simeq
(\Mat_n(\Delta),*)$. This follows from the Noether-Skolem theorem,
the fact that the unitary group $U(\Mat_n(\Delta),*)$ is semi-simple
and simply-connected, and the Kneser theorem on the $H^1$-vanishing for
simply-connected groups over non-Archimedean local fields. For the
details of this argument, see for example \cite[Section 2]{yu:c}. 
Therefore, we may assume that the image of $F$ is fixed by $*$. 
Since the maximal
semi-simple commutative subalgebras of $\Mat_n(\Delta)$ stable by the
involution has degree
$n[K_0:\Q]$ over $\Q$, the image $\iota(F)$ contains the center
$K_0$. This shows that $F$ contains a subfield that is isomorphic to
$K_0$.    
 
As $[F:K_0]=n$, a $K_0$-algebra embedding
$\iota:F\to \Mat_n(\Delta)$ always exists if $\Delta$ is the
endomorphism algebra of a simple abelian variety $A_1$ of dimension
$[K_0:\Q]$. The abelian variety $(A=A_1^n,\iota)$ with RM by $F$ has
endomorphism algebra $\End^0(A)\simeq \Mat_n(\Delta)$.  \\  

\npr (Type IV) The algebra $\Delta$ is a central simple algebra over a
 CM field $K$ with maximal real number field $K_0$. For any
 finite place $v$ of $K$, one has
 \begin{equation}
   \label{eq:31}
   \inv_v(\Delta)+\inv_{\sigma(v)}(\Delta)=0, \quad \text{and}\quad  
  \inv_v(\Delta)=0 \quad \text{if}\quad \sigma(v)=v,  
 \end{equation}
 where $\sigma$ is the non-trivial automorphism of
 $K/K_0$. This is the necessary and sufficient condition 
 for the central simple $\Delta$ that admits a positive involution.
  Let
 $m:=\deg(\Delta)$. The algebra $\Delta$ contains a (maximal) totally
 real subfield $F'$ of degree $m[K_0:\Q]$ over $\Q$, therefore
 one has $m[K_0:\Q]|\dim A_1$, or equivalently $nm[K_0:\Q]\mid g$. 
On the other hand, since the field $F$ can be embedded into
 $\Mat_n(\Delta)$, one has the condition $g|nm[K_0:\Q]$. This shows
 $\dim A_1=m[K_0:\Q]$ and $A_1$ has smCM. Since any maximal 
 semi-simple commutative subalgebra of $\Mat_n(\Delta)$ that is stable
 for a positive
 involution has degree $nm[K_0:\Q]$, which is also equal to $[F:\Q]$, 
 the
 image $\iota(F)$ contains a subfield which is isomorphic to
 $K_0$. Therefore, we may assume that $F$ contains $K_0$ and that 
 the embedding map $\iota$ is a $K_0$-algebra homomorphism. 

 We now determine the condition for $\Delta$ so that the totally real
 field $F$ can be embedded into $\Mat_n(\Delta)$ over $K_0$. Note that 
$[F:K_0]=\deg(\Mat_n(\Delta)/K)$. Put $L=F\otimes_{K_0} K$. The CM field
$L$ is isomorphic to a maximal subfield of $\Mat_n(\Delta)$. The
 local-global principle 
(cf. \cite[Theorem~A.1]{prasad-rapinchuk:metakernel96} and
\cite{yu:embed}) asserts that this holds if and only if 
for any (finite) ramified place $v$ of $K$ for $\Delta$, one has
 $[L_w:K_v]\cdot \inv_v(\Delta)\in \Z$. If $v$ is fixed by $\sigma$,
 then $\inv_v(\Delta)=0$ and hence the condition is satisfied
 automatically. Let $v$ be a finite ramified place of $K$
 and  $v_0$ be the place of $K_0$
 below $v$; the place $v_0$ splits in $K$. 
 For any place $w_0$ of $F$
 lying over $v_0$, one also has that $w_0$ splits in $L$. Therefore
 $[F_{w_0}:K_{0,v_0}]=[L_w:K_v]$, where $w$ is any place of $L$ over
 $w_0$. One concludes that the field $F$ can be embedded into
 $\Mat_n(\Delta)$ 
 over $K_0$ if and only if for any (finite) ramified place $v$ of $K$
 for $\Delta$ (note: $\sigma(v)\neq v$), 
 \begin{equation}
   \label{eq:32}
[F_{w_0}:K_{0,v_0}]\cdot \inv_v(\Delta)=0\quad (\text{in\ }\
\Q/\Z),\quad \forall\, w_0|v_0.      
 \end{equation}

Conversely, suppose we are given a central division algebra 
$\Delta$ over a CM field $K$ with maximal totally real field $K_0$ 
that satisfies the local condition (\ref{eq:31}). Suppose also that 
(1) $K_0$ is contained in $F$, (2) $mn[K_0:\Q]=[F:\Q]$, where
$m=\deg(\Delta/K)$, and (3) the condition (\ref{eq:32}) is satisfied.
Then there is an abelian variety $(A,\iota)$ with RM by $F$ such that
$\End^0(A)\simeq \Mat_n(\Delta)$. Indeed, we take a complex abelian
variety $A_1$ of dimension $m[K_0:\Q]$ such that $\End^0(A_1)\simeq
\Delta$. Above discussion shows that there is a $K_0$-algebra
embedding $\iota: F\to \Mat_n(\Delta)$. Set $A:=A_1^n$, then the pair 
$(A,\iota)$ has the desired property.  \\

We summarize the classification in the following theorem. This is 
a result of Chai \cite[Lemma 6]{chai:ho}, while we make it 
more explicit about simple algebras that can
actually occur as endomorphism algebras of RM abelian 
varieties over more general (unspecified) ground fields. 

\begin{thm}\label{32}
  Let $(A,\iota)$ be a $g$-dimensional abelian variety with RM by $F$
  over a field $k$. Then $\End^0(A)\simeq \Mat_n(\Delta)$ for a
  positive integer $n$ and a division algebra $\Delta$. \\

\npr {\rm (Type I)} The algebra $\Delta=K_0$ is a totally real 
number field. Then the field $K_0$ can be embedded as a subfield in
$F$ with $n[K_0:\Q]=g$. 

Conversely, given a subfield $K_0$ of $F$ of
degree $d$, then there is an abelian variety $(A,\iota)$ with RM by
$F$ such that $\End^0(A)\simeq \Mat_n(K_0)$, where $n:=g/d$. \\

\npr {\rm (Type II)} The algebra $\Delta$ is a totally indefinite
 quaternion algebra over a totally real number field $K_0$. Then the
 field $K_0$ can be embedded as a subfield of $F$ with $2n[K_0:\Q]=g$
 and $F$ splits the quaternion algebra $\Delta$. 

 Conversely, given a
 subfield $K_0$ of $F$ with $2n[K_0:\Q]=g$ for some positive integer
 $n$, and $\Delta$ an indefinite quaternion algebra over $K_0$ such
 that $F$ splits $\Delta$. Then there is an abelian variety
 $(A,\iota)$ with RM by $F$ such that $\End^0(A)\simeq
 \Mat_n(\Delta)$. \\

\npr {\rm (Type III)} The algebra $\Delta$ is a totally definite
 quaternion algebra over a totally real number field $K_0$. Then the
 field $K_0$ can be embedded as a subfield of $F$  with
 $n[K_0:\Q]=g$. The characteristic of the base field $k$ is a prime
 $p>0$, and $A$ is supersingular. Moreover, we have $\Delta\simeq
 D_{p,\infty}\otimes_\Q {K_0}$. 
 Under the assumption of the hypothesis {\bf (H)}, we have
  $\Delta\simeq D_{p,\infty}$ with $K_0=\Q$, or 
 $D_{\infty_1,\infty_2}$ with $K_0=\Q(\sqrt{p})$. 

 Conversely, suppose $(\Delta, K_0)$ is one of the above two cases
 and suppose that $K_0$ is contained in $F$. Then there exists an
 abelian 
 variety $(A,\iota)$ with RM by $F$ over a finite field such that
 $\End^0(A) \simeq \Mat_n(\Delta)$, where $n=g/[K_0:\Q]$. \\

\npr {\rm (Type IV)} 
 The algebra $\Delta$ is a central simple algebra over a
 CM field $K$ with maximal real number field $K_0$. Then the field
 $K_0$ can
 be embedded in $F$ with $g=nm[K_0:\Q]$ where $m=\deg(\Delta/K)$. For
 any finite ramified place $v$ of $K$ 
 for $\Delta$, we have
 \begin{equation}
   \label{eq:33}
[F_{w_0}:K_{0,v_0}]\cdot \inv_v(\Delta)=0\quad (\text{in\ }\
\Q/\Z),\quad 
\forall\, w_0|v_0,      
 \end{equation} 
where $v_0$ is the place of $K_0$ below $v$. 

Conversely, let $\Delta$ be a central division algebra over a CM
field $K$ with maximal totally real field $K_0$ that admits a positive
involution. Suppose that 
(1) $K_0$ is contained in $F$, (2) $mn[K_0:\Q]=[F:\Q]$ for some
positive integer $n$, where $m=\deg(\Delta/K)$, and (3) the condition
(\ref{eq:33}) is satisfied. 
Then there is an abelian variety $(A,\iota)$ with RM by $F$ such that
$\End^0(A)\simeq \Mat_n(\Delta)$.
\end{thm}

\section*{Acknowledgments}
  The author thanks C.-L. Chai for his constant support and
  encouragement, as well as his work \cite{chai:ho} where the second
  part of this paper is based on. 
  The author was partially supported by grants NSC
  100-2628-M-001-006-MY4 and AS-99-CDA-M01. Finally he thanks the
  referee for helpful suggestions on the organization that improve the
  presentation of this paper.

\end{document}